\documentclass[12pt,a4paper,reqno]{amsart}
\usepackage[utf8]{inputenc}

\usepackage{amsmath, amsthm, amssymb, amsfonts, amscd, mathrsfs, xcolor, mathabx, graphicx, caption}
\usepackage[foot]{amsaddr}
\usepackage{hyperref}
\newtheorem{theo}{Theorem}[section]

\newtheorem{prop}[theo]{Proposition}

\newtheorem{rem}[theo]{Remark}

\newcommand{\dbR}{\mathbb{R}} 
\newcommand{\dbC}{\mathbb{C}} 
\newcommand{\dbZ}{\mathbb{Z}} 
\newcommand{\dbHil}{\mathcal{H}} 

\newcommand{\dbol}[1]{\overline{#1}}
\newcommand{\dbwh}[1]{\widehat{#1}}

\title[Sampling in the Euclidean motion group]{Sampling in the Euclidean motion group and a problem from brain's visual cortex}
\author[Davide Barbieri]{Davide Barbieri}
\address{Departamento de Matem\'aticas, Universidad Aut\'onoma de Madrid. Campus Cantoblanco, 28049 Madrid.}
\email{davide.barbieri@uam.es}

\begin{document}

\begin{abstract}
We study a sampling problem for the continuous abstract wavelet transform associated with the quasiregular representation of the $SE(2)$ group, for a modulated gaussian mother wavelet. This problem is motivated by the functional modeling of neuronal activities in brain's primary visual cortex. We provide a characterization in terms of a dual Gramian matrix, and study numerically the relationships among the parameters defining the sampling and the mother wavelet.\vspace{5ex}

\begin{center}
\emph{This paper is dedicated to Carlos Cabrelli,\\ a great mathematician, and a friend.}
\end{center}
\end{abstract}

\maketitle

\section{Introduction}
Let $SE(2) = \dbR^2 \rtimes SO(2) \approx \dbR^2 \rtimes S^1$ be the group of rigid motions of the Euclidean plane, with composition law
\[
(x,\theta)\cdot(y,\alpha) = (x + r_\theta y, \theta + \alpha)
\]
for $x, y \in \dbR^2$ and $\theta, \alpha \in S^1 = [0,2\pi)$, where $r_\theta = \binom{\cos\theta \ \ -\sin\theta}{\sin\theta \ \ \phantom{-}\cos\theta}$. Let $\pi$ denote the unitary representation $\pi : SE(2) \to \mathcal{U}(L^2(\dbR^2))$ defined by
\[
\pi(x,\theta) f(y) = f(r_{-\theta}(y - x)) \, , \ f \in L^2(\dbR^2)
\]
and, for a mother wavelet $\psi \in L^2(\dbR^2)$, let $W_\psi : L^2(\dbR^2) \to \mathcal{C}(\dbR^2 \times S^1)$ be the associated continuous abstract wavelet transform \cite{Fuhr2005}:
\begin{equation}\label{eq:W}
W_\psi f (x,\theta) = \langle f, \pi(x,\theta)\psi\rangle_{L^2(\dbR^2)} = \int_{\dbR^2} f(y) \dbol{\psi(r_{-\theta}(y - x))} dy.
\end{equation}

In this paper we consider a problem of sampling for \eqref{eq:W} with points that lay on a subset of $\dbR^2 \times S^1$ defined as the graph of a function $\Theta : \dbR^2 \to S^1$. More precisely, we want to find conditions that relate a countable subset $S = \{(\lambda, \Theta(\lambda)) : \lambda \in \Lambda \subset \dbR^2\} \subset \dbR^2 \times S^1$, a mother wavelet $\psi$ and a Hilbert subspace $\dbHil \subset L^2(\dbR^2)$ so that there exist $0 < A \leq B < \infty$ such that:
\[
A \|f\|_{L^2(\dbR^2)}^2 \leq \sum_{\lambda \in \Lambda} |W_\psi f(\lambda, \Theta(\lambda))|^2 \leq B \|f\|_{L^2(\dbR^2)}^2 \, , \ \forall \, f \in \dbHil.
\]

This is the condition for the family $\{\pi(\lambda,\Theta(\lambda))\psi\}_{\lambda \in \Lambda}$ to be a frame \cite{Daubechies1992, Christensen2016}, or, equivalently, for the map $f \mapsto \{W_\psi f(\lambda, \Theta(\lambda)) : \lambda \in \Lambda\}$ to be bounded from $\dbHil$ to $\ell_2(\Lambda)$ and invertible with a bounded inverse. The stability of the inversion is quantified (in arithmetic systems with a relative error uniform bound) by the condition number $\kappa = \frac{B}{A}$.

We are particularly interested in a family of mother wavelets $\psi$ parametrized by $p, \sigma > 0$, having the form of modulated gaussians, known as Gabor or Morlet wavelet
\begin{equation}\label{eq:gaussian}
\psi(x) = \frac{1}{2\pi \sigma^2} e^{2\pi i p x_1} e^{-\frac{|x|^2}{2\sigma^2}} \, , \ x \in \dbR^2
\end{equation}
and in a class of maps $\Theta$ with a discrete image and whose level sets $\{\lambda \in \Lambda : \Theta(\lambda) = \theta_0\}$, for $\theta_0 \in \Theta(\Lambda)$, are approximately equally spaced.

\vspace{0.5ex}

The paper is organized as follows. In \S \ref{sec:motivation} we discuss a motivation for this problem, found in visual perception, and provide a brief description of key experimental facts in the neurophysiology of brain's visual cortex and of their mathematical models. In \S \ref{sec:statement} we give a formal statement of the problem, including a feasible Hilbert subspace $\dbHil$, and a class of functions $\Theta$ that is meaningful for the problem in vision. This class is also chosen in order to allow us to use techniques of shift-invariant spaces, in terms of which we obtain a characterization of frames in \S \ref{sec:dualGram} with Theorem \ref{th}. Finally, in \S \ref{sec:numerics}, we study numerically the relationships between $\psi, \Theta, \dbHil$ and the frame constants.

\vspace{0.5ex}

In order to make this paper accessible to readers interested in the mathematical modeling of vision, we have tried to be self-contained. We have also included a dedicated conclusions section, comparing the numerical results with pertinent literature in neuroscience.

\subsection{Notations}
We denote the Fourier transform of $f \in L^1(\dbR^2)$ by
\[
\mathcal{F}(f)(\xi) = \dbwh{f}(\xi) = \int_{\dbR^2} f(x) e^{-2\pi i x.\xi} dx
\]
and we use the same notation for its standard extension to a unitary operator on $L^2(\dbR^2)$.
For a variable $x$ on a Lebesgue measurable space we indicate by $dx$ the Lebesgue measure.
For $x, \xi \in \dbR^2$ we denote by $e_x(\xi) = e^{-2\pi i x.\xi}$ the complex exponential functions and, for $f \in L^2(\dbR^2)$ and $\theta \in S^1$, we use the notations $f_\theta(x) = f(r_{-\theta}x)$ and $\dbwh{f_\theta}(\xi) = \dbwh{f}(r_{-\theta}\xi)$. We will make use of the following elementary relationship:
\begin{equation}\label{eq:intertwining}
\mathcal{F}(\pi(x,\theta) f) = e_x \dbwh{f_\theta}.
\end{equation}

A lattice subgroup $\Gamma < \dbR^2$ is full rank if there exists a $2 \times 2$ invertible matrix $A \in GL_2(\dbR)$ such that $\Gamma = A \dbZ^2 = \{A\binom{n_1}{n_2} : n_i \in \dbZ, i=1,2\}$. A  fundamental set for $\Gamma$ is a measurable $\Omega \subset \dbR^2$ representative of the quotient $\dbR^2/\Gamma$, for example $\Omega = A [0,1]^2 = \{A \binom{\xi_1}{\xi_2} : \xi_i \in [0,1], i=1,2\}$. The annihilator lattice of $\Gamma$ is $\Gamma^\perp := \{ \nu \in \dbR^2 : e_{\nu}(\gamma) = 1 \, \forall \, \gamma \in \Gamma\}$ and, if $\Gamma = A \dbZ^2$ for $A \in GL_2(\dbR)$, it is given by $\Gamma^\perp = (A^t)^{-1}\dbZ^2$.

Let us recall here that, for such $\Gamma, \Gamma^\perp$ and $\Omega$, by a standard change of variable in Parseval's theorem for Fourier series, for all $f \in L^2(\Omega)$ we have (see also \cite[Th. 1.6.1, Th. 2.1.2]{Rudin1962} or \cite[Th. 8.4.2]{Deitmar2005})
\begin{equation}\label{eq:Parseval}
\sum_{\nu \in \Gamma^\perp} \left| \int_\Omega f(\omega) e_\nu(\omega) d\omega \right|^2 = |\Omega| \int_\Omega |f(\omega)|^2 d\omega.
\end{equation}

We will make use of the \emph{bracket map} \cite{BVR1994, HSWW2010}: if $\Gamma < \dbR^2$ is full-rank and $\Omega$ is a fundamental set for $\Gamma^\perp$, for $f, \phi \in L^2(\dbR^2)$ we denote by\vspace{-.5ex}
\begin{equation}\label{eq:bracket}
[f,\phi]_{\Gamma}(\omega) = \sum_{\nu \in \Gamma^\perp} \dbwh{f}(\omega + \nu) \dbol{\dbwh{\phi}(\omega + \nu)} \, ,\vspace{-.7ex}
\end{equation}
which is an $L^1(\Omega)$ function satisfying\vspace{-.65ex}
\begin{equation}\label{eq:bracketintegral}
\langle f, \phi\rangle_{L^2(\dbR^2)} = \int_\Omega [f,\phi]_{\Gamma}(\omega) d\omega.\vspace{-.7ex}
\end{equation}

\subsection{Related works}
The continuous wavelet transform \eqref{eq:W} is a classical object when paired with dilations, see e.g. \cite{AMVT2004}. For anisotropic dilations, its discretization was studied for the so-called \emph{curvelets} \cite{CandesDonoho2005}, and a similar problem replacing rotations by other, possibly more computationally effective, operators led to the development of \emph{wavelets with composite dilations} and \emph{shearlets} \cite{Labate2006, Labate2005}. The $SE(2)$ transform \eqref{eq:W}, its discretization, and its use for image processing, also in relationship with the modeling of the visual cortex, were studied in \cite{Duits2007}.

In this paper, where dilations are not considered, we follow an approach that relies on sampling with multiple copies of a given lattice. These sets were first considered in the setting of \emph{multi-tilings} and Riesz bases of exponentials \cite{Kolountzakis2015}. In particular, we make use of arguments introduced in \cite{AgoraAntezanaCabrelli2015} and, for frames of \emph{sub-(multi)tilings}, in \cite{BHM2017, BCHLMM2018}. The problem we consider here could be seen as a problem of frames of modulates, hence bearing several similarities with those works.

The framework we consider for this problem is that of shift-invariant spaces, and it is largely based on \cite{Bownik2008, CabrelliPaternostro2010}. Shift-invariant spaces together with rotations have been used to study optimal approximation in \cite{BCHM2019, BCHM2020}, while the reconstruction of $SE(2)$-wavelets from arbitrary sampling sets without computing the dual frame was addressed in \cite{B2022}.

The mathematical modeling of neuronal activity in primary visual cortex V1 in terms of operations similar to \eqref{eq:W}, discussed in \S \ref{sec:motivation}, dates back to the 80's \cite{Daugman1985}, and a neurophysiological model that adds scale to \eqref{eq:W} was proposed in \cite{SCP2008}. The first models of the V1 distributions that motivate our construction of sampling sets, again see \S \ref{sec:motivation}, are from the 90's \cite{NW1994}. The pinwheel structure encountered in these distributions motivated a large sequence of modeling works, based on self-organization arguments (see e.g. \cite{Swindale2000, KW2010}), or on geometric arguments, such as the construction based on rotations and scale given in \cite{BCS2018}.

\vspace{1ex}
\noindent
{\bf Acknowledgements}
This project has received funding from Grant PID2022-142202NB-I00 funded by AEI/10.13039/501100011033.

\newpage

\section{Motivation from visual perception}\label{sec:motivation}
This sampling problem is motivated by classical experiments in the neurophysiology of the brain's primary visual cortex (V1) of many mammals, including humans. We introduce here a mathematical model that relates such experiments with the previously described problem.

The activity\footnote{The electric activity of sensory neurons can be approximated by a Poisson process governed by the value of a linear functional of the stimulus \cite{Simoncelli2004}.} of the neurons in V1 - in response to a visual stimulus (image) captured by the retina - is characterized by a sensitivity to a given region in the visual space, called the \emph{classical receptive field}, and by a selectivity to given values in a set of features of that region in the image. In particular, for a V1 neuron, one edge orientation elicits the maximal response, called the \emph{preferred orientation} of the neuron \cite{HubelWiesel1962}.

Assuming for simplicity the visual space to be the plane $\dbR^2$, a visual stimulus can be modeled as a function $f : \dbR^2 \to \dbR$ measuring light intensity. For a given neuron $\eta \in V1$, denote by $x_\eta \in \dbR^2$ the center of its receptive field, and by $\theta_\eta \in S^1$ its preferred orientation. A large set of experiments on different mammals shows that the value $W_\psi f (x_\eta,\theta_\eta)$ computed as in \eqref{eq:W}, with $\psi$ in the form \eqref{eq:gaussian}, provides a meaningful approximation of the linear activity of the neuron $\eta$ in response to an image $f$ for a majority of V1 neurons called \emph{simple cells} \cite{JP1987,  Ringach2002, Keliris2019}. Some of such measurements, and fits with functions of the form $\pi(0,\theta)\psi$, in cats and macaques, are reproduced in Figure \ref{fig:receptivefields}. More realistic models include more parameters, nonlinearities, cortical connectivities, or dynamic behaviors, that we are disregarding (for a nonexhaustive panorama, see \cite{CarandiniHeeger2012, Bertalmio2024, RaoBallard1999, Gerstner2014, B2015} and references therein).

\begin{figure}[h!]
\centering
\includegraphics[height=.285\textheight]{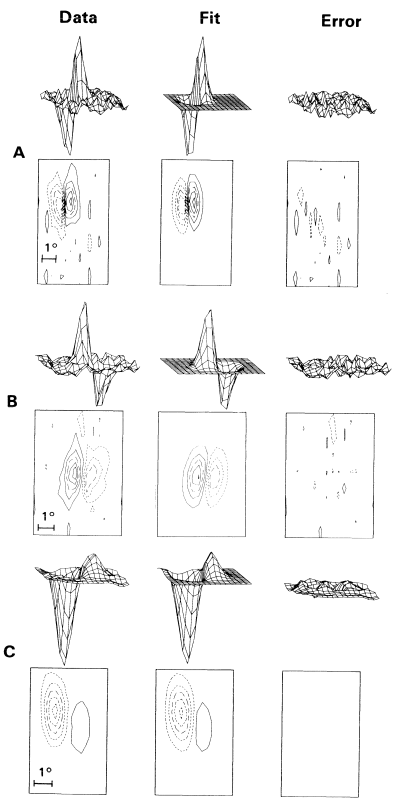} \qquad
\includegraphics[height=.285\textheight]{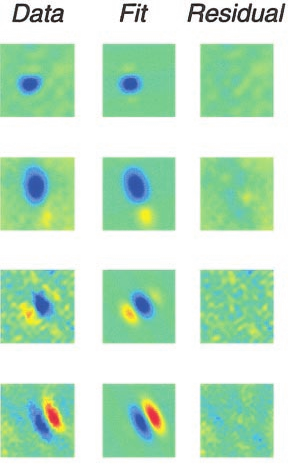} \ \
\includegraphics[height=.285\textheight]{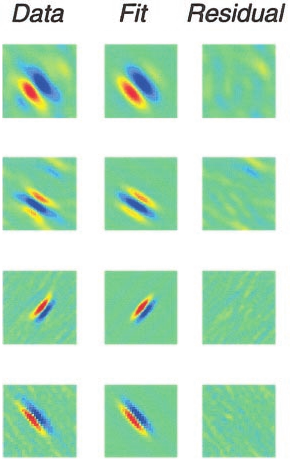}
\caption{Electrophysiological measurements of the representative vectors for the linear functionals that govern the activities of different V1 neurons, and fits with rotated modulated gaussians. Left, from \cite{JP1987}: in cats. Right, from \cite{Ringach2002}: in macaques.}\label{fig:receptivefields}.
\end{figure}

In a second class of experiments, it has been observed that the parameters $\{(x_\eta,\theta_\eta)\}_{\eta \in V_1} \subset R^2 \times S^1$ associated with the preferred positions and orientations of V1 neurons are not distributed in a three dimensional fashion (as the three dimensional volume of brain occupied by V1 may allow). They rather follow two-dimensional distributions on parallel layers of V1. The parameters $x_\eta$, defining the position of the receptive fields in the visual space, are associated with the physical position of neurons in a two dimensional layer of V1 by means of an approximately conformal map\footnote{This is an approximately log-polar map, that associates a larger cortical area to the center of the visual space and a smaller cortical area to the peripheral vision.}, called the \emph{retinotopic map} \cite{Valois1990}. A common simplification is that of taking this map to be the identity, hence identifying a two dimensional layer of V1 with the visual space itself. On the other hand, the preferred orientations $\theta_\eta$ of V1 neurons are distributed according to a two dimensional \emph{orientation preference map} (OPM) that is the same on each two dimensional layer \cite{Ohki2006}. This map has a characteristic correlation length \cite{NW1994}, meaning that similar preferred orientations are encountered in neurons that are approximately equidistant, and a pinwheel-like shape \cite{BS1986}, see Figure \ref{fig:OPM}.

\begin{figure}[h!]
\centering
\includegraphics[height=.15\textheight]{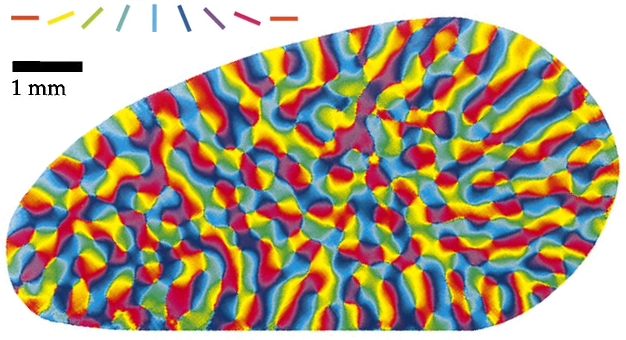} \
\includegraphics[height=.15\textheight]{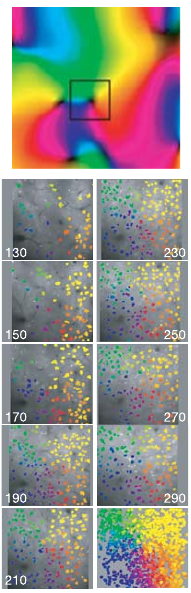} \
\includegraphics[height=.15\textheight]{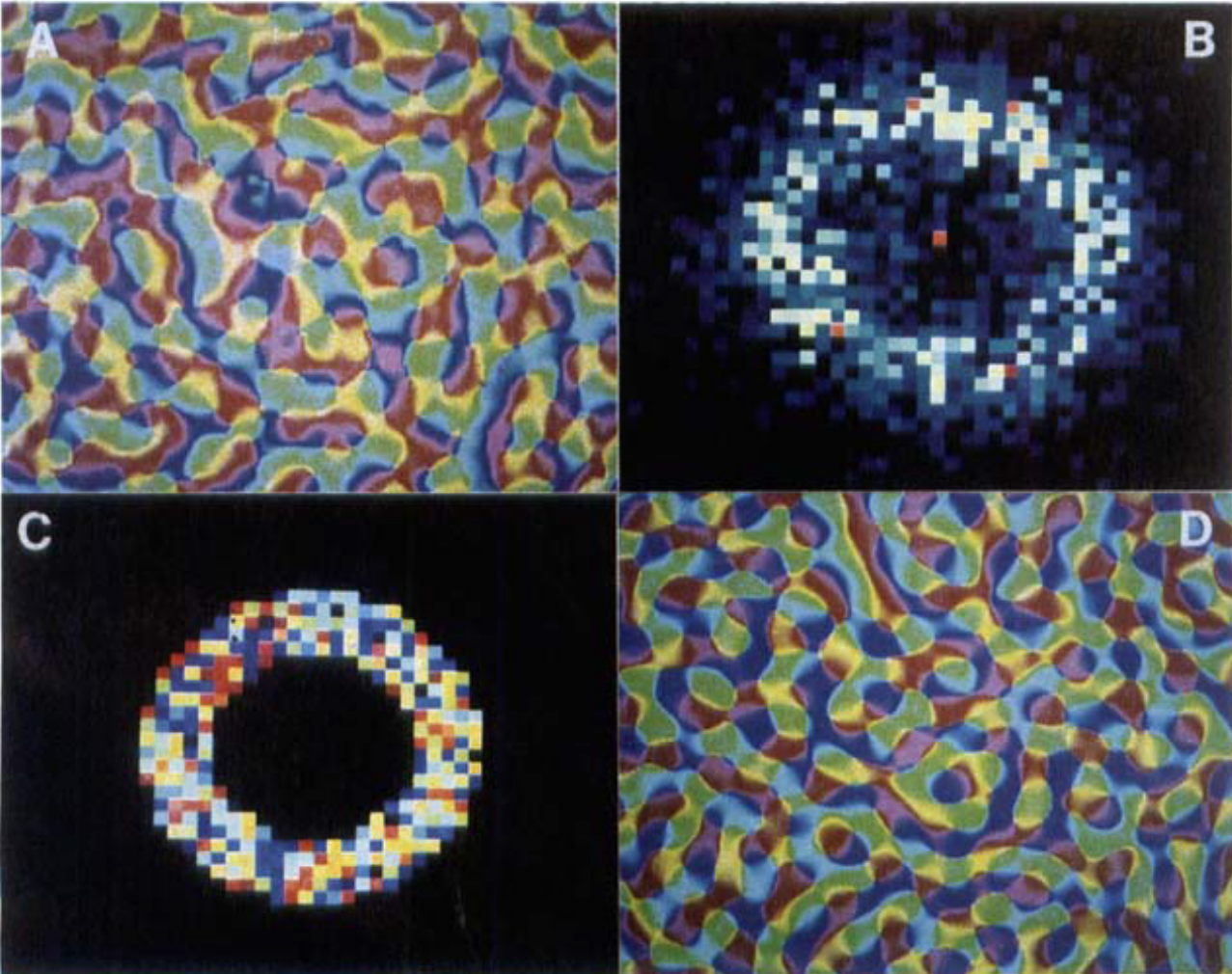}
\caption{Left, from \cite{Bosking1997}: optical imaging measurements of OPM in V1 of a tree-shrew. Color indicates the preferred orientation of the underlying neuron.
Center, from \cite{Ohki2006}: preferred orientation of single neurons measured at different widths, proving the two-dimensional structure of OPM.
Right, from \cite{NW1994}: (A) image of OPM in macaques measured by \cite{BS1986}, (B) its Discrete Fourier Transform, (D) the image that results by inverse DFT of (C) a random distribution on a thin annulus.}\label{fig:OPM}
\end{figure}
\vspace{-1ex}
Let us then denote by $\Theta : \dbR^2 \to S^1$ the function that assigns a preferred orientation to each point of $\dbR^2$, representing the location of a visual neuron on a two dimensional layer of V1. After identifying two dimensional layers of V1 with the visual space by the mentioned simplified retinotopy assumption, we can summarize the model of linear V1 neural activity at a point $x$ on a layer of V1 in response to a stimulus $f \in L^2(\dbR^2)$, as
\begin{equation}\label{eq:Wrestricted}
W_\psi f (x,\Theta(x)) = \frac{1}{2\pi \sigma^2} \int_{\dbR^2} f(y) e^{-2\pi i p \langle \binom{\cos\Theta(x)}{\sin\Theta(x)} , y-x\rangle} e^{-\frac{|y - x|^2}{2\sigma^2}} dy.
\end{equation}
Since V1 neurons form a discrete set, a natural problem for this simplified model is to describe which visual stimuli can be represented by a discretization of this activity, and how reliable this representation is.

This is the problem of frames defined in the introduction, that of sampling of the coefficients of the $SE(2)$ transform with a Gabor mother wavelet, restricted to the surface of preferred orientations defined by an OPM $\Theta$, as in \eqref{eq:Wrestricted}. The main intuitive difficulty, for the realization of such a condition, that the OPM shown in Figure \ref{fig:OPM} may introduce, is that a given preferred orientation is available only at a discrete, almost equally spaced, set of points. Such spacing is the correlation length of the OPM: the minimal distance between two neurons with the same preferred orientation. If this is too large, that orientation may not be detectable at the spatial positions between them. This problem seems to be intrinsically related to three quantities: the correlation length of the OPM, the \emph{size} $\sigma$ of the receptive fields, and the variability of input signals: if the ratio between the first two quantities is large, then information may be lost in images with large variability. We will see in \S \ref{sec:Calderon} that the natural way to quantify this variability is in terms of the bandwidth of natural images that are expected to be processed, since they need to belong to a Paley-Wiener space. On the other hand, the relationship between the first two quantities alone was experimentally explored in \cite{Bosking2002}, finding that in brain's visual cortex they are not very different. This fact is closely related to a principle that has been addressed by the neuroscience community to understand the structure of the visual cortex since the 90's, called the \emph{coverage principle}: ``the cortical representation has to cover as good as possible the visual stimulus features'' (see \cite{Swindale1991, KeilWolf2011} and references therein).

Two other quantities that intutively may affect the frame condition are the shape of the receptive field, and the number of different angles defined by $\Theta$. The first quantity can be characterized as the product of $\sigma$ by the spatial frequency parameter $p$, representing the effective number of wave oscillations under the gaussian bell. This product is not constant on real visual cortices, and its distribution seems to be related to the optimization of the angular uncertainty in the measurement of local orientations \cite{BCS2014}. However, we will consider it constant, and its role for the problem of sampling in $SE(2)$ will become apparent in \S \ref{sec:numerics}. The second quantity may not seem to be so critical, since one may expect angles represented by a biological tissue to cover an essentially continuous range of possibilities. It is, however, a recurring question in computational biology that of how many angles are actually implemented in the visual cortex. We will treat the problem of discrete sampling with special sets, that actually express only a finite number of angles. These sets, introduced in \S \ref{sec:statement}, are built to attack directly the intuitive difficulty described above: they represent sets of neurons associated with the same angle, more or less equally spaced. Thus, the number of angles will be somewhat related to the spatial density of neurons and, as such, given the other parameters, there will be a minimal number of angles necessary to obtain the frame condition.

\newpage

\section{Sampling conditions}

We begin the section by introducing a natural Hilbert subspace of $L^2(\dbR^2)$ for sampling the $SE(2)$ wavelets we are interested in. We then formalize the problem, by defining a class of sampling sets that is compatible with the presented model of V1. This class allows us to work within the setting of shift-invariant spaces: the result is largely inspired by the works \cite{Bownik2008, CabrelliPaternostro2010, AgoraAntezanaCabrelli2015}, in that it provides a characterization in terms of the spectrum of the so-called \emph{dual Gramian}, and includes techniques of bracket maps from \cite{BHP2020}. We conclude the section with a result for a frequency cutoff of the mother wavelet \eqref{eq:gaussian} since, even if it is more elementary and has limited applications, it provides sufficient conditions that are explicit and thus useful for heuristics on parameter choices.

\subsection{A Hilbert space for a frame condition}\label{sec:Calderon}
It is well-known that, for any $\psi \in L^2(\dbR^2)$, there exist no constants $0 < A \leq B < \infty$ such that
\begin{equation}\label{eq:continuousframe}
A \|f\|^2_{L^2(\dbR^2)} \leq \int_{S^1} \int_{\dbR^2} |W_\psi f (x,\theta)|^2 dx d\theta \leq B \|f\|^2_{L^2(\dbR^2)}
\end{equation}
holds for all $f \in L^2(\dbR^2)$ (see \cite{WW2001}). This can be seen explicitly as follows: by Plancherel's theorem and \eqref{eq:intertwining}, we have that
\[
\int_{S^1} \int_{\dbR^2} |W_\psi f (x,\theta)|^2 dx = \int_{\dbR^2} |\dbwh{f}(\xi)|^2 \left(\int_{S^1} |\dbwh{\psi_\theta}(\xi)|^2 d\theta\right)d\xi \, ,
\]
so \eqref{eq:continuousframe} is equivalent to the Calder\'on-type condition 
\begin{equation}\label{eq:Calderon}
A \leq \int_{S^1} |\dbwh{\psi_\theta}(\xi)|^2 d\theta \leq B
\end{equation}
for a.e. $\xi \in \dbR^2$. However, since $\|\psi\|_{L^2(\dbR^2)}^2 \! = \! \displaystyle\int_0^\infty \!\!\int_{S^1}\left|\dbwh{\psi}(\varrho r_\theta \hat{v})\right|^2 \! d\theta \varrho d\varrho$ for any $v \in \dbR^2$ with $|v| = 1$, then \eqref{eq:Calderon} can not hold for a $\psi \in L^2(\dbR^2)$ and for a.e. $\xi \in \dbR^2$ unless $\psi = 0$.\vspace{0.5ex}

On the other hand, if $\psi \in L^2(\dbR^2)$ is such that \eqref{eq:Calderon} holds for a.e. $\xi \in B(0,\varrho) := \{\xi \in \dbR^2 : |\xi| < \varrho\}$, for a finite $\varrho > 0$, then \eqref{eq:continuousframe} holds for all $f$ in the Paley-Wiener Space
\[
\textsc{pw}(B(0,\varrho)) := \{f \in L^2(\dbR^2) : \textnormal{supp}(\dbwh{f}) \subset B(0,\varrho)\}.
\]
Even if the modulated gaussian \eqref{eq:gaussian} does not belong to $\textsc{pw}(B(0,\varrho))$, since its Fourier transform reads $\displaystyle \dbwh{\psi}(\xi) = e^{-2\pi^2\sigma^2|\xi - \binom{p}{0}|^2}$, we have
\[
\int_{S^1} |\dbwh{\psi_\theta}(\xi)|^2 d\theta = 2\pi e^{-4\pi^2\sigma^2(|\xi|^2 + p^2)} I_0(8\pi^2\sigma^2 p|\xi|)
\]
where $I_0$ is the modified Bessel function of the the first kind, of order zero. This implies that, for any finite $\varrho > 0$, such a $\psi$ satisfies Calder\'on's condition \eqref{eq:Calderon} for all $\xi \in B(0,\varrho)$.
Thus, since $W_\psi f$ is a continuous function, it is meaningful to consider, for $\psi$ as in \eqref{eq:gaussian}, the problem of discrete sampling of $W_\psi f$ when $f \in \textsc{pw}(B(0,\varrho))$.

\newpage
\subsection{Statement of the problem}\label{sec:statement}

Given a full rank lattice $\Gamma < \dbR^2$ we consider sampling sets of the form\vspace{-.5ex}
\[
S = \bigsqcup_{k = 1}^N \Big((\Gamma + \alpha_k) \times \{\theta_k\}\Big)\vspace{-.5ex}
\]
for $N$ points of the plane $\{\alpha_k\}_{k = 1}^{N} \subset \dbR^2$ such that
\[
(\Gamma + \alpha_k) \cap (\Gamma + \alpha_{k'}) = \emptyset \, , \ \textnormal{for } k \neq k'
\]
and for $N$ distinct angles $\{\theta_k\}_{k = 1}^{N} \subset [0, 2\pi)$. Such an $S$ can be written
\[
S = \{(\lambda, \Theta(\lambda)) : \lambda \in \Lambda \subset \dbR^2\} \subset \dbR^2 \times S^1
\]
by choosing $\Lambda$ to be the countable set\vspace{-.5ex}
\begin{equation}\label{eq:Lambda}
\Lambda = \bigsqcup_{k = 1}^N (\Gamma + \alpha_k)\vspace{-.5ex}
\end{equation}
and $\Theta : \Lambda \to [0,2\pi)$ to be the function defined by
\[
\Theta(\lambda) = \theta_k \ \textnormal{if } \lambda \in \Gamma + \alpha_k \, , \ \textnormal{for } k = 1, \dots, N.
\]
For $\psi$ as in \eqref{eq:gaussian} and $I_\psi$ defined by
\begin{equation}\label{eq:I}
I_\psi(f) = \sum_{\lambda \in \Lambda} |W_\psi f(\lambda, \Theta(\lambda))|^2 \, , \ f \in \textsc{pw}(B(0,\varrho))\vspace{-.5ex}
\end{equation}
we want to characterize the frame condition
\begin{equation}\label{eq:newframe}
A \|f\|^2_{L^2(\dbR^2)} \leq I_\psi(f) \leq B \|f\|^2_{L^2(\dbR^2)} \, , \ \forall \, f \in \textsc{pw}(B(0,\varrho)).
\end{equation}

\vspace{.5ex}
From the point of view of the mathematical modeling of V1, this construction does not reproduce exactly what was measured in OPM, but it preserves the key property of having evenly spaced points associated with the same angle. As an additional observation, the distribution of each preferred orientation in V1 indeed appears to be compatible with a perturbation of a full rank lattice, see Figure \ref{fig:pincenters}.

\begin{figure}[h!]
\centering
\includegraphics[width=.74\textwidth]{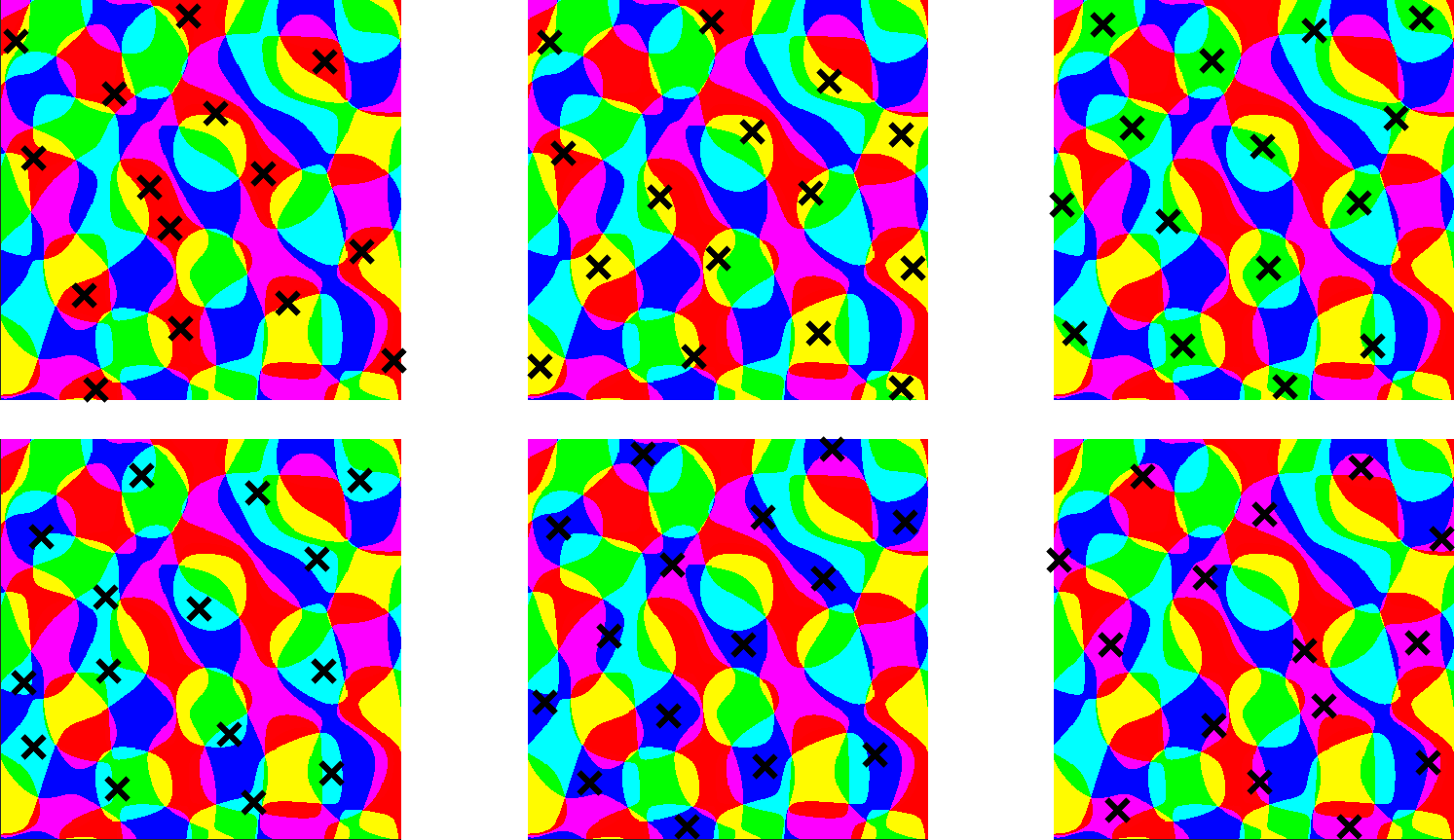}
\caption{Artificial OPM-like distribution generated as in Figure \ref{fig:OPM} Right, D \cite{NW1994}.\\ In this image, colors (representing preferred orientations) are quantized into 6 channels. For each, the centers of a standard k-means clustering are marked with an X.}\label{fig:pincenters}
\end{figure}

\newpage
\subsection{Characterization in terms of a dual Gramian matrix}\label{sec:dualGram}

Here we prove a characterization of the frame condition \eqref{eq:newframe} in terms of an equivalent spectral condition that is analogous to \cite[Th. 2.5]{Bownik2008} and \cite[Th. 4.1]{CabrelliPaternostro2010}. Since our shift-invariant space is a Paley-Wiener space, the result has close similarities with \cite[Th. 2.9]{AgoraAntezanaCabrelli2015} and \cite[Prop. 3.1]{BCHLMM2018}. We provide a direct proof of the equivalence using an idea from \cite[Th. 29]{BHP2020}.

\begin{theo}\label{th}
Let $\Gamma < \dbR^2$ be a full-rank lattice, let $\Gamma^\perp$ be its annihilator lattice, and let $\Omega \subset \dbR^2$ be a fundamental set for the quotient $\dbR^2 / \Gamma^\perp$. For $\varrho > 0$, denote by
\[
\mathcal{V}_\varrho(\omega) = \{\nu \in \Gamma^\perp : |\nu + \omega| < \varrho\}
\]

Let $p, \sigma > 0$, and let $\{(\alpha_k,\theta_k)\}_{k = 1}^N \subset \dbR^2 \times S^1$. For $\omega \in \Omega$, denote by $G(\omega) \in \dbC^{\mathcal{V}_\varrho(\omega) \times \mathcal{V}_\varrho(\omega)}$ the matrix
\begin{equation}\label{eq:dualGram}
G(\omega)_{\nu, \nu'} = e^{-\pi^2\sigma^2|\nu - \nu'|^2} \sum_{k = 1}^N e_{\alpha_k}(\nu - \nu') e^{-4\pi^2\sigma^2 \big|\omega + \frac{\nu + \nu'}{2} - \binom{p\cos\theta_k}{p\sin\theta_k}\big|^2}
\end{equation}
and let $\Sigma(\omega)$ be the spectrum of $G(\omega)$. Then \eqref{eq:newframe} holds for two constants $0 < A \leq B < \infty$ if and only if
\begin{equation}\label{eq:spectrum}
\Sigma(\omega) \subset \left[\frac{A}{|\Omega|},\frac{B}{|\Omega|}\right] \ \textnormal{a.e. } \omega \in \Omega.
\end{equation}
\end{theo}

\begin{proof}
\emph{Step I}: we claim that the quantity \eqref{eq:I} can be written, using the notation \eqref{eq:bracket}, and denoting by $\phi_k = \pi(\alpha_k,\theta_k)\psi$, as
\[
I_\psi(f) = |\Omega| \int_\Omega \sum_{k = 1}^N \left| [f,\phi_k]_{\Gamma}(\omega) \right|^2 d\omega .
\]
In order to see this, let us first observe that, by Plancherel's theorem and the relation \eqref{eq:intertwining}, and with the notation $\psi_\theta(x) = \psi(r_{-\theta}(x))$, it holds
\[
I_\psi(f) = \! \sum_{k = 1}^N \sum_{\gamma \in \Gamma} |\langle f , \pi(\gamma + \alpha_k, \theta_k) \psi\rangle_{L^2(\dbR^2)}|^2
 = \! \sum_{k = 1}^N \sum_{\gamma \in \Gamma} |\langle \dbwh{f} , e_{\gamma + \alpha_k} \dbwh{\psi_{\theta_k}}\rangle_{L^2(\dbR^2)}|^2.
\]
Here, by periodization and the definition of annihilator lattice, we have
\begin{align*}
\langle \dbwh{f} , e_{\gamma + \alpha_k} \dbwh{\psi_{\theta_k}}&\rangle_{L^2(\dbR^2)} = \int_\Omega \sum_{\nu \in \Gamma^\perp} \dbwh{f}(\omega + \nu) \dbol{e_{\gamma + \alpha_k}(\omega + \nu) \dbwh{\psi_{\theta_k}}(\omega + \nu)} d\omega\\
& = \int_\Omega \left(\sum_{\nu \in \Gamma^\perp} \dbwh{f}(\omega + \nu) \dbol{e_{\alpha_k}(\omega + \nu) \dbwh{\psi_{\theta_k}}(\omega + \nu)}\right) e_{-\gamma}(\omega) d\omega.
\end{align*}
So, calling $\dbwh{\phi_k} = e_{\alpha_k}\dbwh{\psi_{\theta_k}}$, we can write
\[
I_\psi(f) = \sum_{k = 1}^N \sum_{\gamma \in \Gamma} \left| \int_\Omega [f,\phi_k]_{\Gamma}(\omega) e_{-\gamma}(\omega) d\omega \right|^2.
\]
The claim follows then by Parseval's theorem for Fourier series \eqref{eq:Parseval}.

\emph{Step II}: we now provide a direct proof that the inequalities
\begin{equation}\label{eq:integralineq}
A \int_\Omega [f,f]_\Gamma(\omega) d\omega \leq I_\psi(f) \leq B \int_\Omega [f,f]_\Gamma(\omega) d\omega
\end{equation}
with $0 < A \leq B < \infty$ hold for all $f \in \textsc{pw}(B(0,\varrho))$, if and only if the inequalities
\begin{equation}\label{eq:nonintegralineq}
\frac{A}{|\Omega|} [f,f]_\Gamma(\omega) \leq \sum_{k = 1}^N \left| [f,\phi_k]_\Gamma(\omega)\right|^2 \leq \frac{B}{|\Omega|} [f,f]_\Gamma(\omega)
\end{equation}
hold for all $f \in \textsc{pw}(B(0,\varrho))$ and almost all $\omega \in \Omega$. Since \eqref{eq:nonintegralineq} implies \eqref{eq:integralineq}, we just need to prove the other implication. Let then \eqref{eq:integralineq} hold and assume, by contradiction, that the left hand side inequality in \eqref{eq:nonintegralineq} does not hold for a nonzero $f \in \textsc{pw}(B(0,\varrho))$. Let
\[
h(\omega) = \sum_{k = 1}^N \left| [f,\phi_k]_\Gamma(\omega)\right|^2 - \frac{A}{|\Omega|} [f,f]_\Gamma(\omega)
\]
and denote by
\[
H = \{\omega \in \Omega : h(\omega) < 0\}.
\]
By the contradiction hypothesis, $H$ has positive measure. Now, let
\[
H^{\Gamma^\perp}_\varrho = B(0,\varrho) \cap \bigsqcup_{\nu \in \Gamma^\perp} (H + \nu)
\]
where the union is disjoint because $H \subset\Omega$, and observe that, since $f \in \textsc{pw}(B(0,\varrho))$, for any $\phi \in L^2(\dbR^2)$ we have
\begin{align*}
\chi_H(\omega) & [f,\phi]_\Gamma(\omega) = \sum_{\nu \in \Gamma^\perp} \chi_H(\omega)\dbwh{f}(\omega + \nu) \dbol{\dbwh{\phi}(\omega + \nu)}\\
& = \sum_{\nu \in \Gamma^\perp} \chi_{H^{\Gamma^\perp}_\varrho}(\omega + \nu)\dbwh{f}(\omega + \nu) \dbol{\dbwh{\phi}(\omega + \nu)} = [f_H,\phi]_\Gamma(\omega)
\end{align*}
where $f_H : = \mathcal{F}^{-1}\big(\chi_{H^{\Gamma^\perp}_\varrho} \dbwh{f}\big)$. Thus, by the contradiction hypothesis
\begin{align*}
0 & > \int_H h(\omega)d\omega = \int_H \left(\sum_{k = 1}^N \Big|[f,\phi_k]_\Gamma(\omega)\Big|^2 - \frac{A}{|\Omega|} [f,f]_\Gamma(\omega)\right) d\omega\\
& = \int_\Omega \left(\sum_{k = 1}^N \Big|\chi_H(\omega) [f,\phi_k]_\Gamma(\omega)\Big|^2 - \frac{A}{|\Omega|} \chi_H(\omega) [f,f]_\Gamma(\omega) \chi_H(\omega)\right) d\omega\\
& = \int_\Omega \sum_{k = 1}^N \left| [f_H,\phi_k]_\Gamma(\omega)\right|^2 d\omega - \frac{A}{|\Omega|} \int_\Omega [f_H,f_H]_\Gamma(\omega) d\omega\\
& =  \frac{1}{|\Omega|} \left(I_\psi(f_H) - A \int_\Omega [f_H,f_H]_\Gamma(\omega) d\omega\right)
\end{align*}
where the last identity was proved in Step I.
This contradicts \eqref{eq:integralineq} because $f_H$ belongs to $\textsc{pw}(B(0,\varrho))$, since $H^{\Gamma^\perp}_\varrho \subset B(0,\varrho)$. A similar argument applies to the right hand side inequality.

\emph{Step III}: in order to complete the proof we only need to show that if, for a fixed $\omega \in \Omega$, we denote by $z \in \dbC^{\mathcal{V}_\varrho(\omega)}$ the vector with components $z_\nu = \dbwh{f}(\omega + \nu)$, with $\nu \in \mathcal{V}_\varrho(\omega)$, then\vspace{-.5ex}
\[
\sum_{k = 1}^N \left| [f,\phi_k]_\Gamma(\omega) \right|^2 = \langle z, G(\omega) z\rangle_{\dbC^{\mathcal{V}(\omega)}}.\vspace{-.5ex}
\]
This is done by direct computation:
\begin{align*}
\sum_{k = 1}^N & \left| [f,\phi_k]_\Gamma(\omega) \right|^2 = \sum_{k = 1}^N \left| \sum_{\nu \in \Gamma^\perp} \dbwh{f}(\omega + \nu) \dbol{\dbwh{\phi_k}(\omega + \nu)}\right|^2\\
& = \sum_{k = 1}^N \sum_{\nu, \nu' \in \Gamma^\perp} \dbwh{f}(\omega + \nu)\dbol{\dbwh{f}(\omega + \nu')} \dbol{\dbwh{\phi_k}(\omega + \nu)}\dbwh{\phi_k}(\omega + \nu')\\
& =  \sum_{\nu \in \Gamma^\perp} \dbwh{f}(\omega + \nu) \left(\sum_{\nu' \in \Gamma^\perp} \left(\sum_{k = 1}^N\dbol{\dbwh{\phi_k}(\omega + \nu)}\dbwh{\phi_k}(\omega + \nu')\right) \dbol{\dbwh{f}(\omega + \nu')}\right)\\
& = \sum_{\nu \in \mathcal{V}(\omega)} \dbwh{f}(\omega + \nu) \left(\sum_{\nu' \in \mathcal{V}(\omega)} G(\omega)_{\nu,\nu'} \dbol{\dbwh{f}(\omega + \nu')}\right) = \langle z, G(\omega) z\rangle_{\dbC^{\mathcal{V}(\omega)}}.
\end{align*}
where the sums are reduced to $\mathcal{V}(\omega)$ because $f \in \textsc{pw}(B(0,\varrho))$, and
\begin{align*}
G(\omega)_{\nu,\nu'} & = \sum_{k = 1}^N \dbol{\dbwh{\phi}_k(\omega + \nu)}\dbwh{\phi}_k(\omega + \nu')\\
& = \sum_{k = 1}^N e_{\alpha_k}(\nu - \nu') e^{-2\pi^2\sigma^2\Big(|\omega + \nu - \binom{p\cos\theta_k}{p\sin\theta_k}|^2 + |\omega + \nu' - \binom{p\cos\theta_k}{p\sin\theta_k}|^2\Big)}.
\end{align*}
Thus, \eqref{eq:nonintegralineq} is the same as \eqref{eq:spectrum}, while the left and right hand sides of \eqref{eq:integralineq} are multiples of the square norm of $f$, because of \eqref{eq:bracketintegral}. This completes the proof that \eqref{eq:newframe} is equivalent to \eqref{eq:spectrum}.
\end{proof}

The characterization of frames in terms of condition \eqref{eq:spectrum} does not provide an explicit condition relating the parameters $p, \sigma$, the values $\{(\alpha_k, \theta_k)\}_{k = 1}^N$ and the size of the lattice $\Gamma$ with the frame constants $A, B$. However, it reduces the question of whether a certain set of values for the sampling problem defines a frame to the evaluation of the spectra of the collection of finite dimensional matrices $G(\omega)$, as $\omega$ varies in $\Omega$. These spectra, and in particular the maximum and the minimum eigenvalues of $G(\omega)$, can be evaluated numerically, and this will be done in \S \ref{sec:numerics} for some meaningful cases. However, at this stage, it is not clear how to find parameters that may lead to frames, and what to expect from the corresponding condition numbers
\[
\kappa = \frac{B}{A} = \frac{\max_{\omega \in \Omega}\max(\Sigma(\omega))}{\min_{\omega \in \Omega}\min(\Sigma(\omega))}.
\]
To get insights on this, we now consider a similar, but simpler, problem.

\newpage
\subsection{Sufficient conditions for frequency cutoff}\label{sec:simpler}

Suppose that, instead of a mother wavelet \eqref{eq:gaussian}, we consider a mother wavelet given by
\begin{equation}\label{eq:simplifiedgaussian}
\dbwh{\varphi}(\xi) = e^{-2\pi^2\sigma^2|\xi - \binom{p}{0}|^2} \chi_{B_L}(\xi)
\end{equation}
for $B_L = \displaystyle B\left(\binom{p}{0},\frac{L}{2}\right) = \left\{\xi \in \dbR^2 : \left|\xi - \binom{p}{0}\right| < \frac{L}{2} \right\}$.

\vspace{0.5ex}
With this frequency cutoff, we can rely on a Parseval frame of exponentials on the support of each rotate $\dbwh{\varphi_{\theta_k}}$ of $\dbwh{\varphi}$. A sufficient condition for \eqref{eq:newframe} can then be obtained by requiring these supports to cover the whole $B(0,\varrho)$. This elementary technique provides explicit results for a problem that can be close enough to the original problem (see also Figure \ref{fig:croppedGabor}) to provide useful heuristics for the numerical simulations.

\vspace{0.5ex}
For $\theta \in S^1$ let $Q_\theta$ be the square circumscribed to the support of $\dbwh{\varphi_\theta}$:
\[
Q_\theta : = \displaystyle\left\{ \xi \in \dbR^2 : \Big\|\xi - \binom{p\cos\theta}{p\sin\theta}\Big\|_\infty < \frac{L}{2} \right\}.
\]
This is a square of side length $L$, so an orthonormal basis for $L^2(Q_\theta)$ is given by $\displaystyle\left\{\frac{1}{L}e_\gamma \, : \, \gamma \in \frac{1}{L}\dbZ^2\right\}$. By Parseval's theorem, for all $f \in L^2(\dbR^2)$, all $\alpha \in \dbR^2$, and all $\theta \in S^1$ we can write
\begin{align*}
\int_{Q_\theta} |\dbwh{f}(\xi) \dbol{\dbwh{\varphi_\theta}(\xi)}|^2 d\xi & = \|e_{-\alpha}\dbwh{f} \dbol{\dbwh{\varphi_\theta}}\|_{L^2(Q_\theta)}^2 = \frac{1}{L^2} \sum_{\gamma \in \frac{1}{L}\dbZ^2} |\langle e_{-\alpha}\dbwh{f} \dbol{\dbwh{\varphi_\theta}}, e_{\gamma} \rangle_{L^2(Q_\theta)}|^2\\
& = \frac{1}{L^2} \sum_{\gamma \in \frac{1}{L}\dbZ^2} |\langle \dbwh{f}, e_{\gamma + \alpha} \dbwh{\varphi_\theta}\rangle_{L^2(\dbR^2)}|^2.
\end{align*}
Thus, if we choose $N$ shifts and angles $\{(\alpha_k, \theta_k)\}_{k = 1}^N \subset \dbR^2 \times S^1$, and we consider $f \in \textsc{PW}(B(0,\varrho))$, we get
\begin{align*}
I_\varphi(f) & = \sum_{k = 1}^N \sum_{\gamma \in \frac{1}{L} \dbZ^2} |\langle \dbwh{f}, e_{\gamma + \alpha_k} \dbwh{\varphi_{\theta_k}}\rangle_{L^2(\dbR^2)}|^2 = L^2 \sum_{k = 1}^N \int_{Q_{\theta_k}} |\dbwh{f}(\xi)|^2 |\dbwh{\varphi_{\theta_k}}(\xi)|^2 d\xi\\
& = L^2 \sum_{k = 1}^N \int_{B(0,\varrho) \cap B\Big(\binom{p\cos\theta_k}{p\sin\theta_k},\frac{L}{2}\Big)} |\dbwh{f}(\xi)|^2 |\dbwh{\varphi_{\theta_k}}(\xi)|^2 d\xi
\end{align*}
where the last identity is just making explicit the support of $\dbwh{f}\dbwh{\varphi_{\theta_k}}$. Now,
since the maximum and the minimum values of $|\dbwh{\varphi_{\theta_k}}(\xi)|^2$ on the domain of integration $D_k := B(0,\varrho) \cap B\Big(\binom{p\cos\theta_k}{p\sin\theta_k},\frac{L}{2}\Big)$ satisfy
\[
\min_{\xi \in D_k} |\dbwh{\varphi_{\theta_k}}(\xi)|^2 \geq e^{-L^2\pi^2\sigma^2} \, , \ \max_{\xi \in D_k} |\dbwh{\varphi_{\theta_k}}(\xi)|^2 \leq 1 \, ,
\]
we get (in a similar fashion to \cite[Cor. 4]{BHM2017})
\[
L^2 e^{-L^2\pi^2\sigma^2} \sum_{k = 1}^N \int_{D_k} |\dbwh{f}(\xi)|^2 d\xi \leq I_\varphi(f) \leq L^2 \sum_{k = 1}^N \int_{D_k} |\dbwh{f}(\xi)|^2 d\xi.
\]
This is sufficient to obtain a frame condition whenever the domains $\{D_k\}_{k = 1}^n$ cover the support of $f$, that we can state as follows.
\begin{prop}\label{prop:simpler}
Let $p, \sigma, L, \varrho > 0$, let $\{\theta_k\}_{k = 1}^N \subset S^1$ and let $m, M \in \dbZ$ be such that
\begin{equation}\label{eq:covering}
m \leq \sum_{k = 1}^N \chi_{B\Big(\binom{p\cos\theta_k}{p\sin\theta_k},\frac{L}{2}\Big)}(\xi) \leq M \quad \textnormal{a.e. } \ \xi \in B(0,\varrho).
\end{equation}
Let $\varphi$ be as in \eqref{eq:simplifiedgaussian}, let $\Gamma = \frac{1}{L}\dbZ^2$ and let $I_\varphi$ be as in \eqref{eq:I}.
Then, for all $f \in \textsc{PW}(B(0,\varrho))$ and for all $\{\alpha_k\}_{k = 1}^N \subset \dbR^2$ it holds
\[
m L^2 e^{-L^2\pi^2\sigma^2} \|f\|_{L^2(\dbR^2)}^2 \leq I_\varphi(f) \leq M L^2 \|f\|_{L^2(\dbR^2)}^2.
\]
In particular, if $m > 0$, then \eqref{eq:newframe} holds with $\frac{B}{A} \leq \frac{M}{m} e^{L^2\pi^2\sigma^2}$.
\end{prop}

\begin{rem}\label{rem:alpha}
The shifts $\{\alpha_k\}_{k =1}^N$ play no role in this sufficient condition for frames. They have been added artificially at the beginning, but are not necessary at any stage.
\end{rem}
\begin{rem}\label{rem:heuristics}
Condition \eqref{eq:covering} is a covering condition that ideally replaces Calder\'on's condition \eqref{eq:Calderon}, but has some intrinsic limitations. In order to have $m > 0$ at $\xi = 0$, $L$ must be at least $2p$, as in Figure \ref{fig:croppedGabor}, while $m > 0$ at $|\xi| = \varrho$ requires $\varrho < p + \frac{L}{2}$, this being an inadmissible value for $\varrho$ even in the limit $N \to \infty$ since then $M$ would diverge.
\end{rem}
\begin{rem}
For the motivating problem in vision, since the receptive field of a neuron is physically compactly supported, its Fourier transform can't be, so $\varphi$ does not provide a completely consistent model. From the practical point of view of fitting the measurements reported in Figure \ref{fig:receptivefields}, $L$ should be large enough to preserve a shape that is similar to a modulated gaussian (see Figure \ref{fig:croppedGabor} as an illustrative example).
\end{rem}

\begin{figure}[h!]
\centering
\includegraphics[width=.78\textwidth]{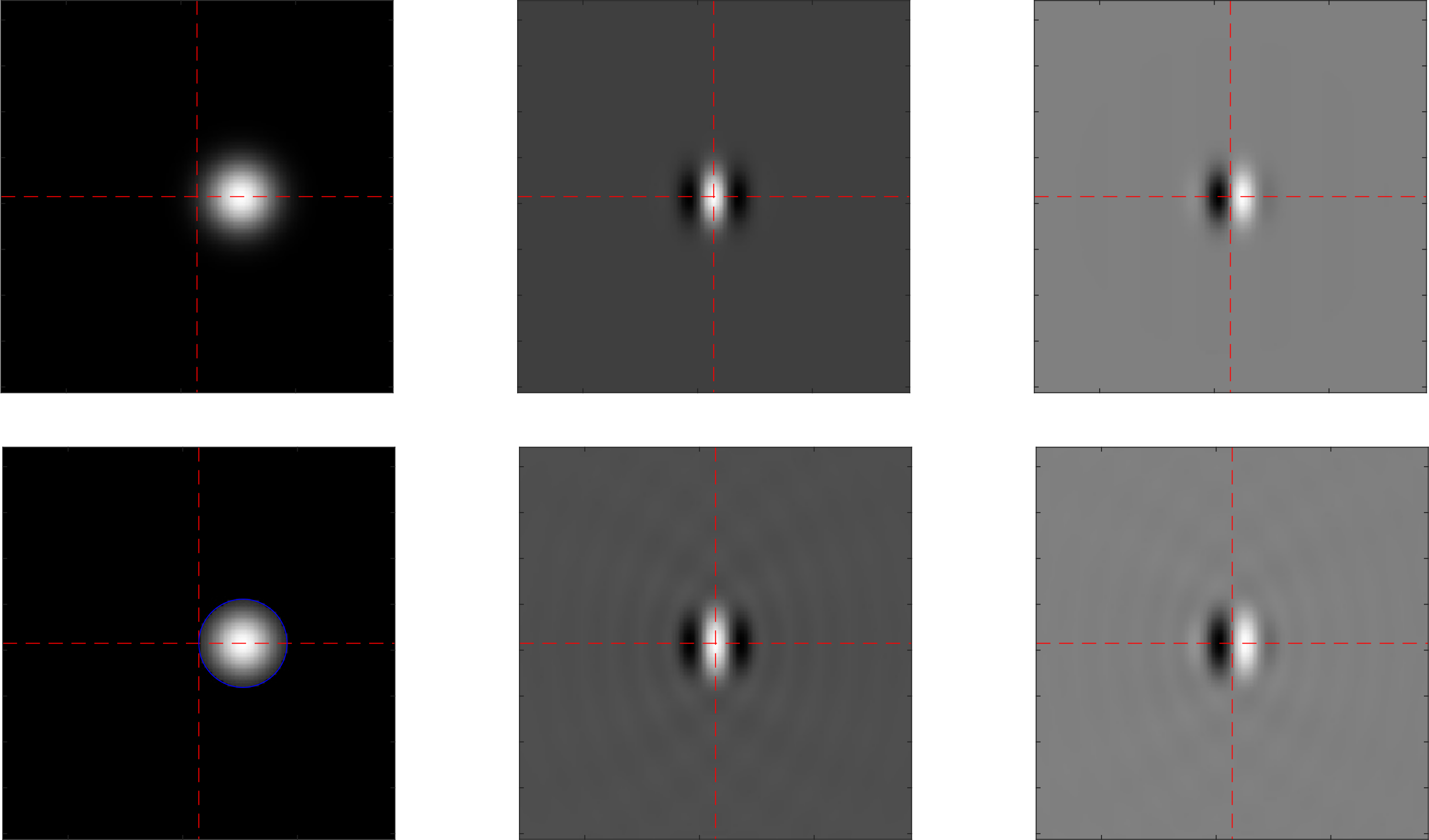}
\caption{Top: Fourier transform, real and imaginary part of $\psi$ with $p \sigma = 0.3$. Bottom: Fourier transform, real and imaginary part of $\varphi$ with $L = 2p$ (tangent to the axes, and retaining about 80\% of the mass of the original gaussian).}\label{fig:croppedGabor}
\end{figure}

\section{Numerical exploration of the parameters space}\label{sec:numerics}

In this section we present numerical results for the spectra of the matrices $G(\omega)$ from \eqref{eq:dualGram}, with $\omega \in \Omega$. We have chosen to consider only the lattice $\Gamma = \dbZ^2$ which, in the units of \S \ref{sec:simpler}, corresponds to $L = 1$, and whose annihilator is again $\Gamma^\perp = \dbZ^2$. For simplicity, we have also considered  $\Omega$ to be the unit square centered at the origin $\Omega = [-\frac12, \frac12]^2$.\\
Since we have no reason to sample $S^1$ in a non-uniform way, we propose numerics for equally spaced angles $\{\theta_k\}_{k = 1}^N$. Moreover, since we have no theoretical argument to tune the lattice shifts $\{\alpha_k\}_{k = 1}^N$, the computed spectra are averaged over 20 repetitions of independent random choices for each of such shifts from a uniform distribution in $(0,1)$.

\vspace{.5ex}

The tested parameters are thus $p, \sigma$, that define the mother wavelet, $\varrho$, that defines the Paley-Wiener space, and the number of angles $N$. The purpose of this numerical analysis is to relate them with the frame condition \eqref{eq:newframe} through the characterization provided by \eqref{eq:spectrum}.

\vspace{.5ex}

The central object is the spectrum of $G(\omega)$, and, in particular,\vspace{-.5ex}
\[
A(\omega) = \min \Sigma(\omega) \, , \ B(\omega) = \max \Sigma(\omega).\vspace{-.5ex}
\]
Observe also that, since the rank of $G(\omega)$ is at most $N$, for \eqref{eq:spectrum} to hold with $A > 0$, the number of angles $N$ must satisfy\vspace{-.6ex}
\[
N \geq \max_{\omega \in \Omega} n(w)\vspace{-.6ex}
\]
where $n(\omega) = |\mathcal{V}(\omega)|$ is the number of lattice points inside a ball of center $\omega$ and radius $\varrho$.
For these reasons, we show the average of $A(\omega), B(\omega)$ over the described randomization of shifts, together with the number $n(\omega)$, as $\omega$ varies in $\Omega$.

\vspace{.5ex}

For each experiment, in order to visualize the parameters $p, \sigma$, we show the real and imaginary parts of $\psi$ in a scale-free figure, thus representing only the shape of $\psi$: this is characterized by the number of oscillations in a standard deviation, or equivalently by the product $p\sigma$. This quantity was actually measured in V1: notably \cite{Ringach2002} shows that, in macaques, whilst there is a more represented shape corresponding to $p\sigma \approx 0.315$, different shapes are present (see also \cite{BCS2014,  B2015}).

\vspace{.5ex}

For the visualization of the Paley-Wiener space in comparison to the other quantities, we also show the semidiscrete Calder\'on function\vspace{-.6ex}
\[
C_\psi(\xi) = \sum_{k = 1}^N |\dbwh{\psi_{\theta_k}}(\xi)|^2\vspace{-.6ex}
\]
and its reciprocal, constrained to the Paley-Wiener ball $\displaystyle \frac{\chi_{B(0,\varrho)}(\xi)}{C_\psi(\xi)}$. This allows one to quantify the condition number for a semidiscrete $SE(2)$ transform with just angle discretization, that behaves similarly to what seen in \S \ref{sec:Calderon}, hence providing simultaneous information on $\varrho, \psi$ and $N$.

\vspace{.5ex}
In all experiments, $A(\omega), B(\omega), C_\psi$ and $\chi_B/C_\psi$ are plotted in logarithmic scale with base 10, and $\Omega$ is sampled into $256 \times 256$ points.

\vspace{.5ex}

To find parameters that produce frames, we start by the heuristics provided by Proposition \ref{prop:simpler} and Remark \ref{rem:heuristics}. In Figure \ref{fig:SIM1} we consider $p = \frac{L}{2} = 0.5$, with a shape factor $p\sigma = \frac{1}{\pi}$. If we take $N = 4$ angles, condition \eqref{eq:covering} shows that, for $m > 0$, one needs $\varrho \leq \frac{1}{\sqrt{2}}$. Choosing this limiting value, by Proposition \ref{prop:simpler} one may expect a condition number $\kappa \leq 2e^4$. Indeed, we obtain $\kappa \approx 51$, which is about half that number.\vspace{-1ex}
\begin{figure}[h!]
\centering
\includegraphics[width=.68\textwidth]{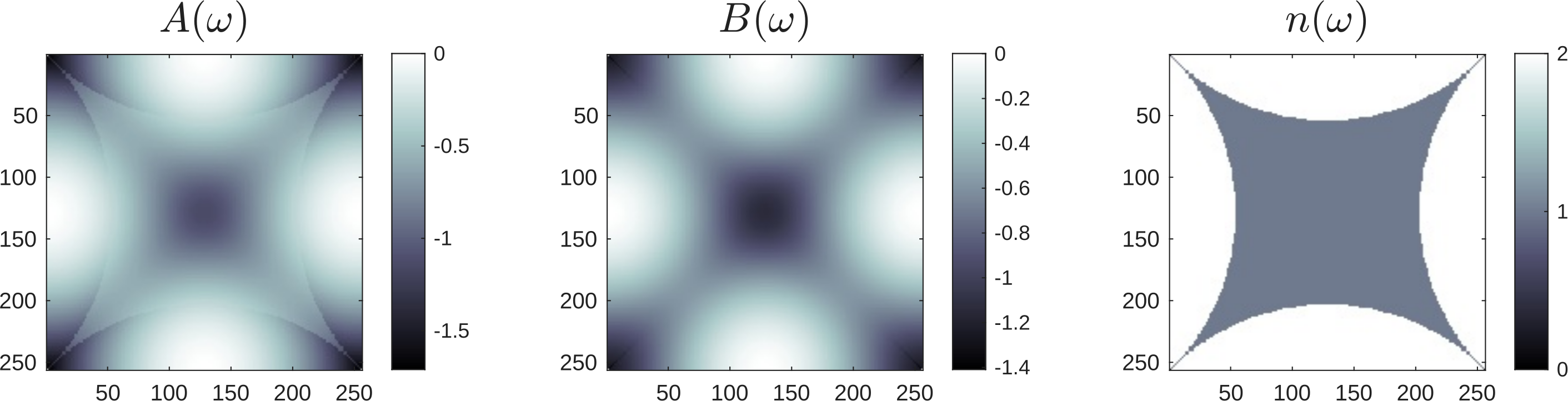}
\includegraphics[width=.68\textwidth]{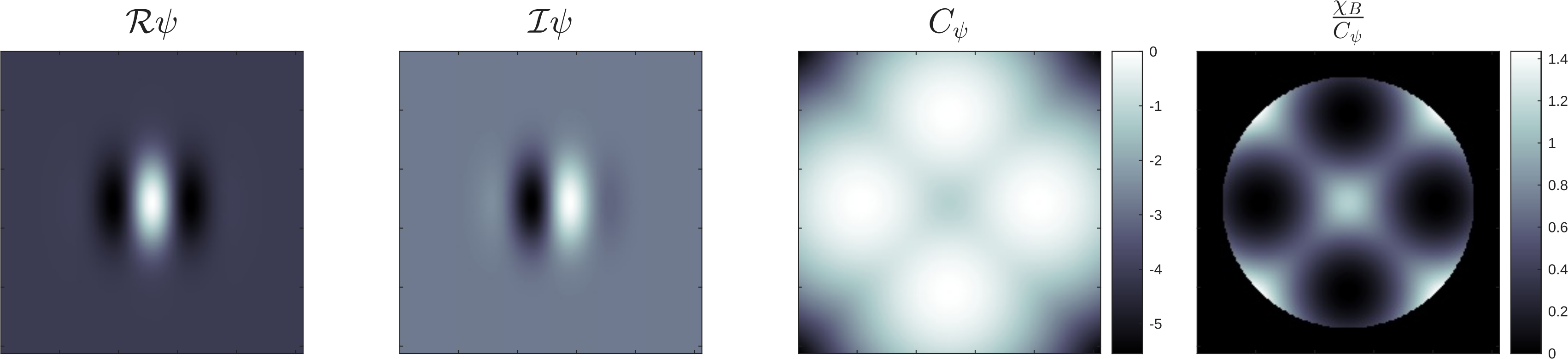}
\caption{Numerical analysis for $p = 0.5, \sigma = 2/\pi, \varrho = 1/\sqrt{2}, N = 4$.}\label{fig:SIM1}
\end{figure}

\vspace{-1.5ex}
The previous values are conservative: the same mother wavelet with $N = 7$ angles produces a frame for $\varrho = 1$ (an unachievable limit for Remark \ref{rem:heuristics}) with a condition number of $\kappa \approx 184$, see Figure \ref{fig:SIM2}.\vspace{-1ex}
\begin{figure}[h!]
\centering
\includegraphics[width=.68\textwidth]{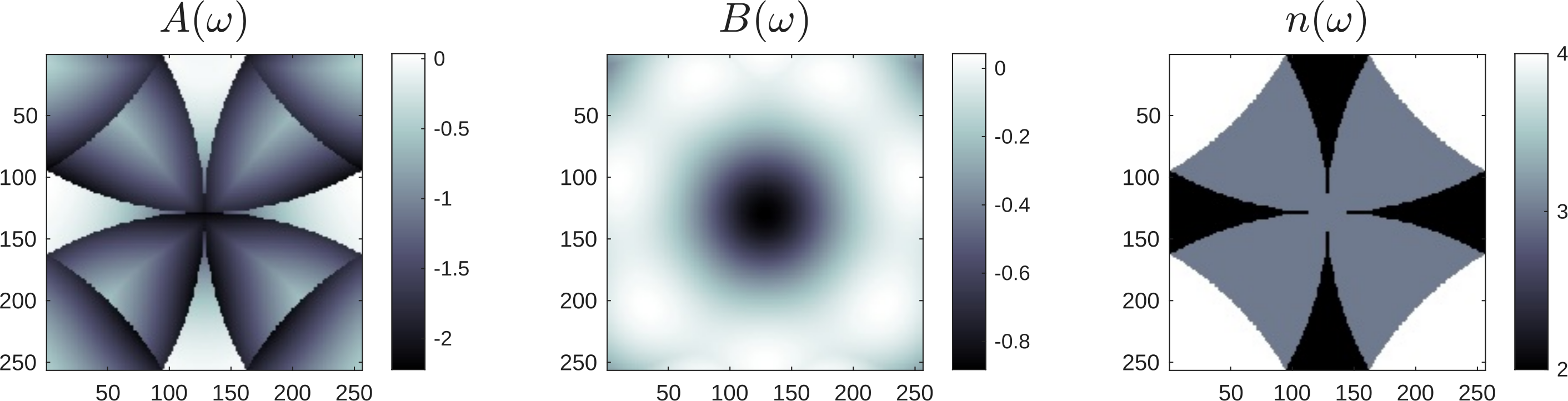}
\includegraphics[width=.68\textwidth]{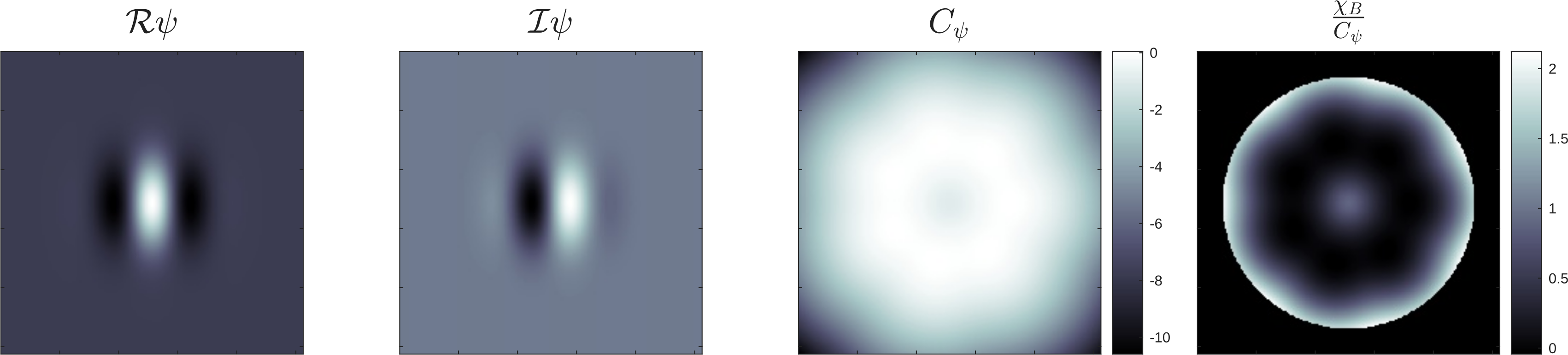}
\caption{Numerical analysis for $p = 0.5, \sigma = 2/\pi, \varrho = 1, N = 7$.}\label{fig:SIM2}
\end{figure}

\vspace{-1.5ex}
Again the same mother wavelet can give a frame for the even larger Paley-Wiener space $\varrho = 1.618$, where $\max_{\omega \in \Omega} n(\omega) = 12$. Using $N = 14$ angles, one gets a finite, but very high, $\kappa \approx 2\cdot 10^9$, see Figure \ref{fig:SIM3}.\vspace{-1ex}
\begin{figure}[h!]
\centering
\includegraphics[width=.68\textwidth]{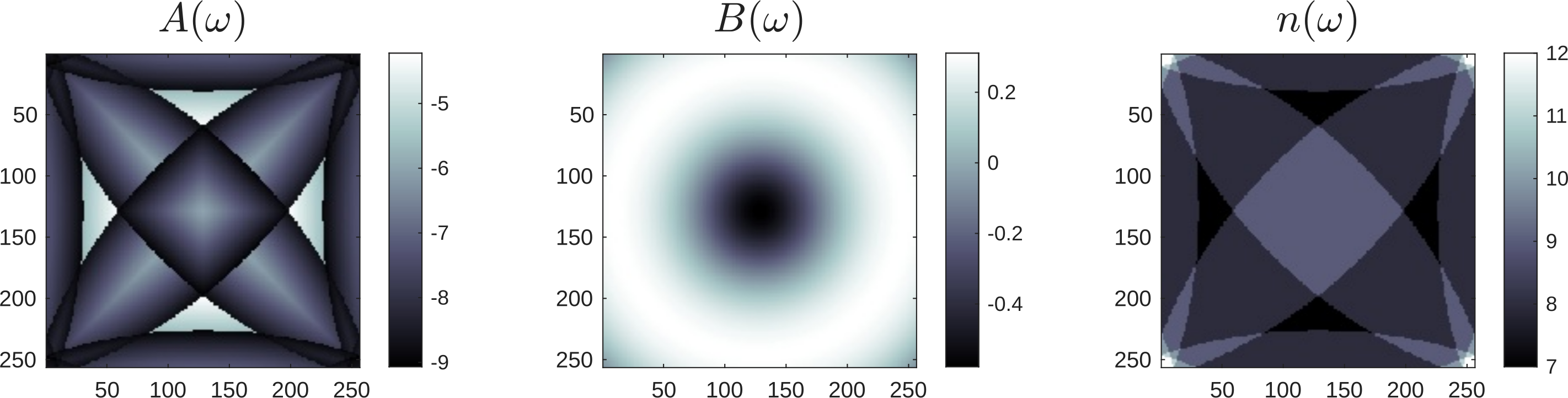}
\includegraphics[width=.68\textwidth]{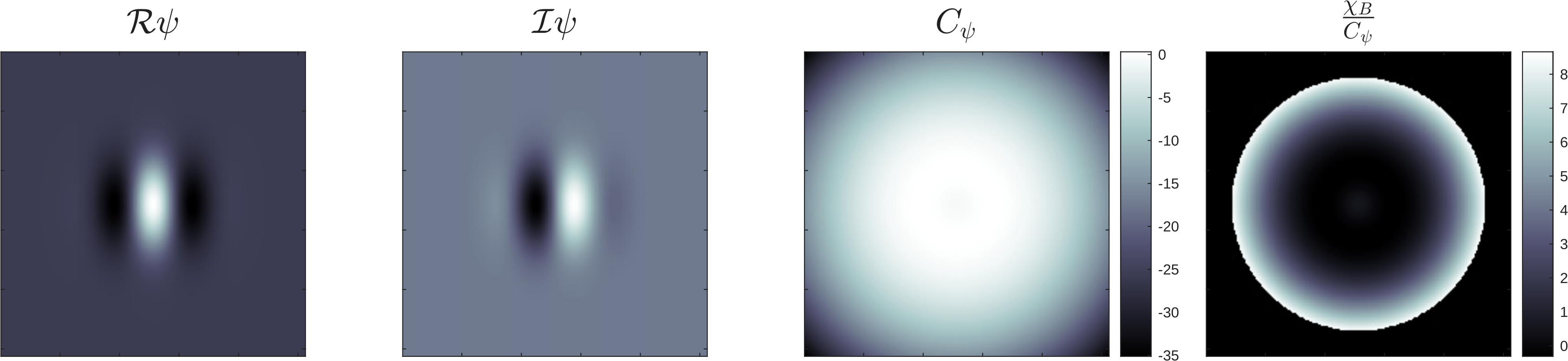}
\caption{Numerical analysis for $p = 0.5, \sigma = 2/\pi, \varrho = 1.618, N = 14$.}\label{fig:SIM3}
\end{figure}

\newpage

One can also move away from the condition $p \leq \frac{L}{2}$ and still get a frame: with $p = 0.7$, and $p\sigma = 0.2$, for the same $\varrho = 1.618$ and $N = 14$ as before, one actually gets a much lower $\kappa \approx 426$, see Figure \ref{fig:SIM4}.
\begin{figure}[h!]
\centering
\includegraphics[width=.68\textwidth]{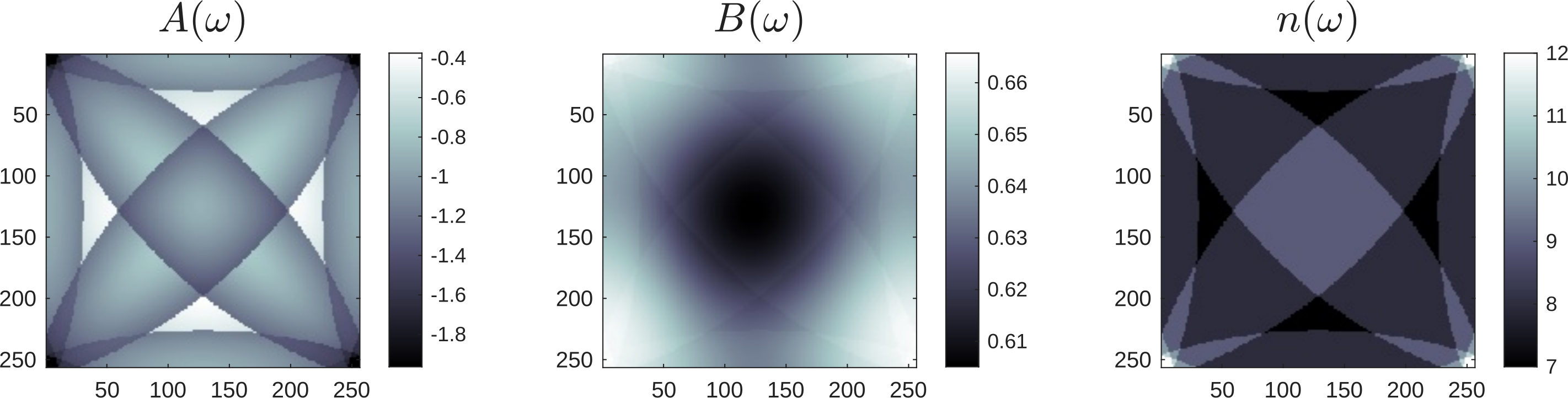}
\includegraphics[width=.68\textwidth]{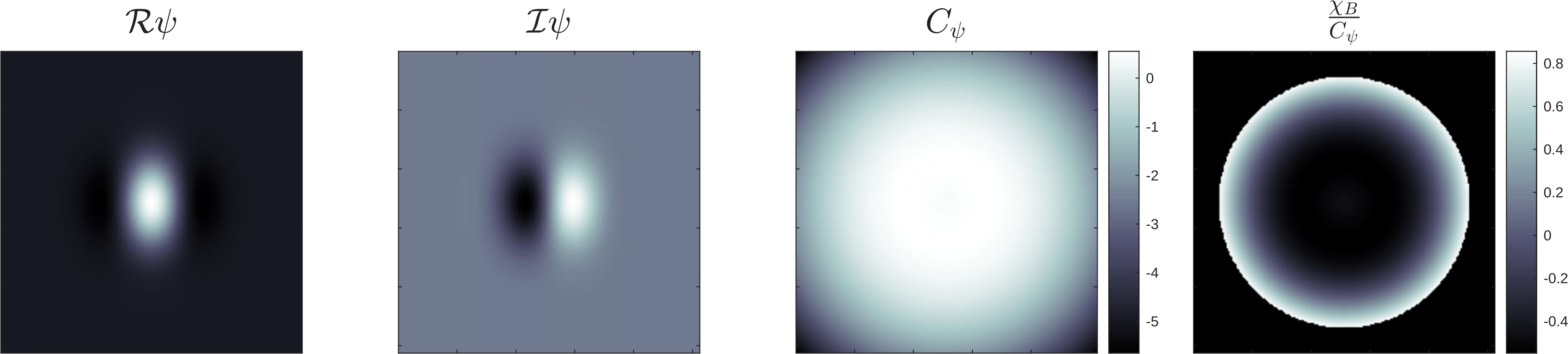}
\caption{Numerical analysis for $p = 0.7, \sigma \approx 0.29, \varrho = 1.618, N = 14$.}\label{fig:SIM4}
\end{figure}

For $p = 0.75$ and the same shape factor $p\sigma = 0.2$, one can get frames for $\varrho = 2$. Since here $\max_{\omega \in \Omega} n(\omega) = 14$, we have used $N = 18$ and obtained $\kappa \approx 1.6\cdot 10^3$, see Figure \ref{fig:SIM5}.
\begin{figure}[h!]
\centering
\includegraphics[width=.68\textwidth]{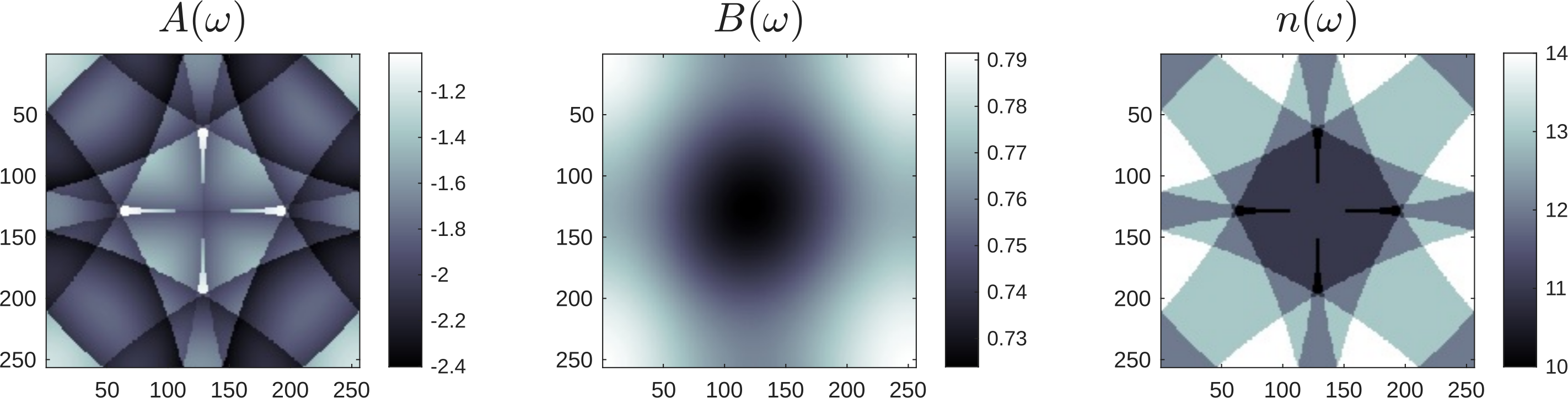}
\includegraphics[width=.68\textwidth]{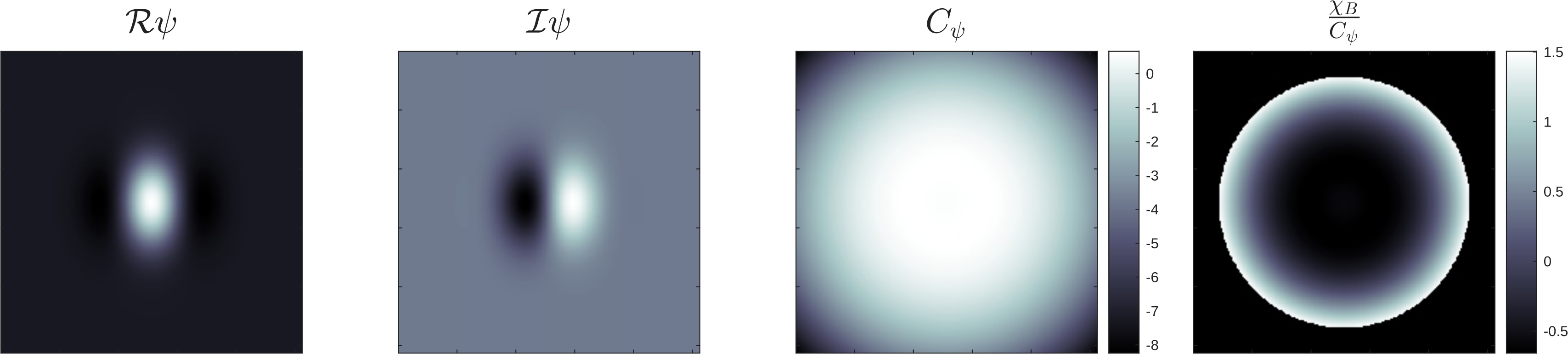}
\caption{Numerical analysis for $p = 0.75, \sigma \approx 0.27, \varrho = 2, N = 18$.}\label{fig:SIM5}
\end{figure}

A frame for a larger shape factor $p\sigma = 0.5$ is obtained with $p = 0.8$ and $N = 12$ on the Paley-Wiener space for $\varrho = \sqrt{2}$, see Figure \ref{fig:SIM5}. This produces a condition number $\kappa \approx 3.8\cdot 10^3$.
\begin{figure}[h!]
\centering
\includegraphics[width=.68\textwidth]{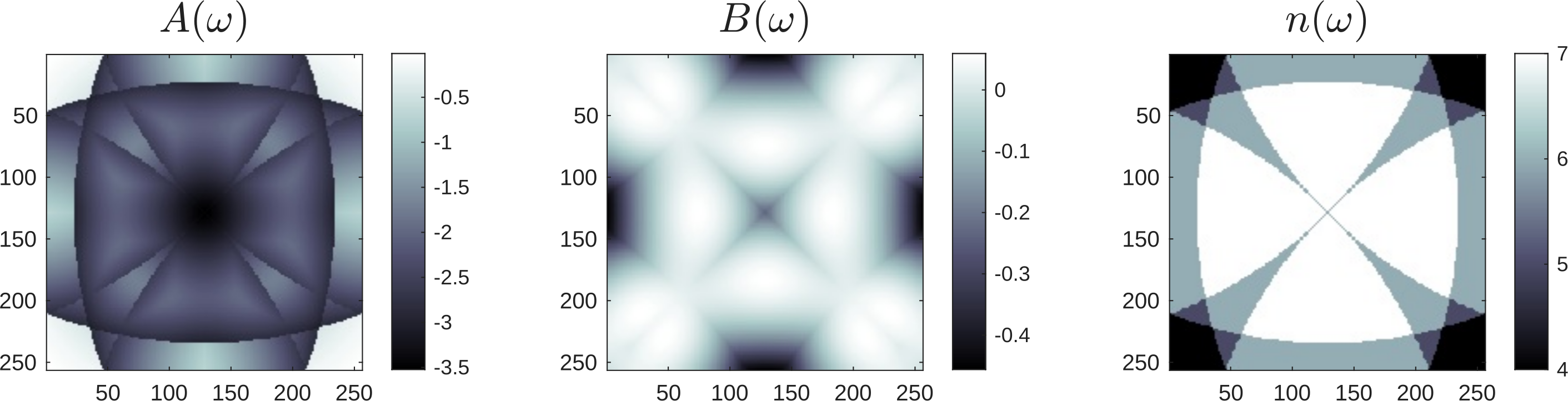}
\includegraphics[width=.68\textwidth]{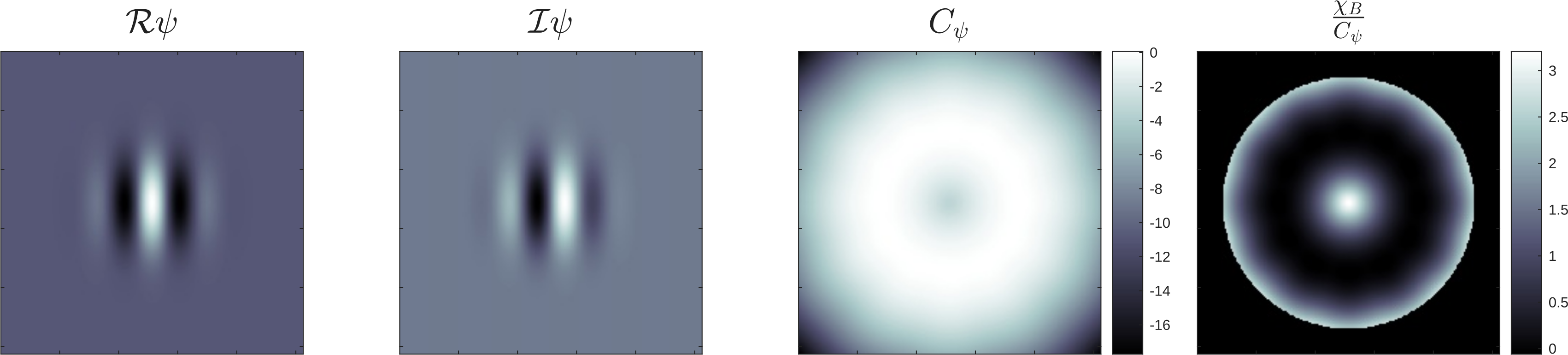}
\caption{Numerical analysis for $p = 0.8, \sigma = 0.625, \varrho = \sqrt{2}, N = 12$.}\label{fig:SIM6}
\end{figure}

\newpage

As a last set of parameters, we show in Figure \ref{fig:SIM7} a regime that differs substantially from Proposition \ref{prop:simpler}, with $p = 1.4$, to fill the Paley-Wiener space for $\varrho = 3$. With a shape factor of $p\sigma = 0.315$ and $N = 100$ angles, one obtains a condition number $\kappa \approx 526$. Here the semidiscrete Calder\'on function is similar to other cases even for lower values of $N$, but we have observed that the condition number lowers by increasing $N$ well above the threshold $\max_{\omega \in \Omega}n(\omega) = 32$.
\begin{figure}[h!]
\centering
\includegraphics[width=.68\textwidth]{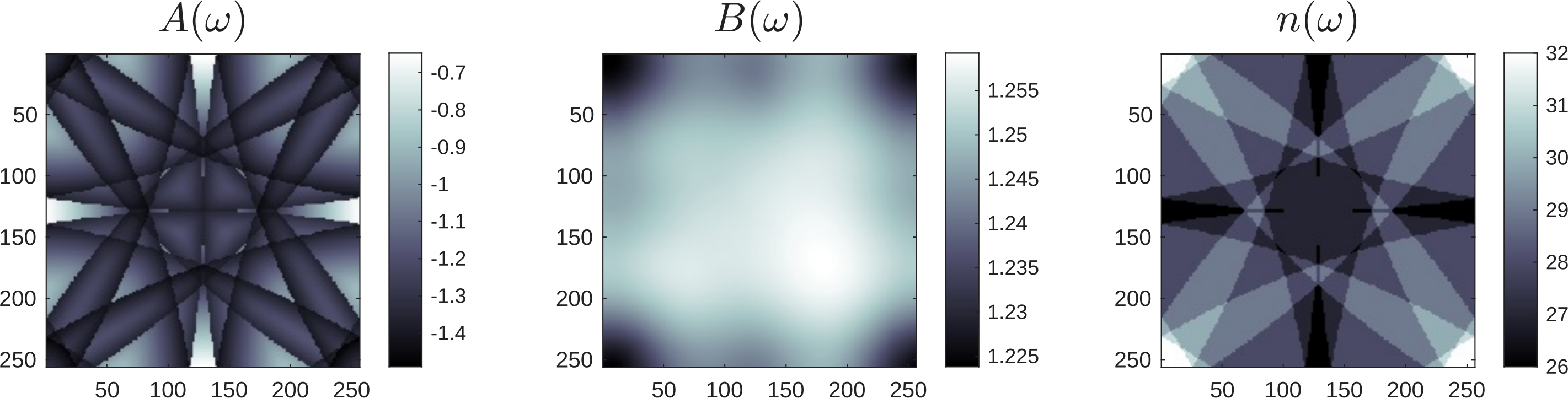}
\includegraphics[width=.68\textwidth]{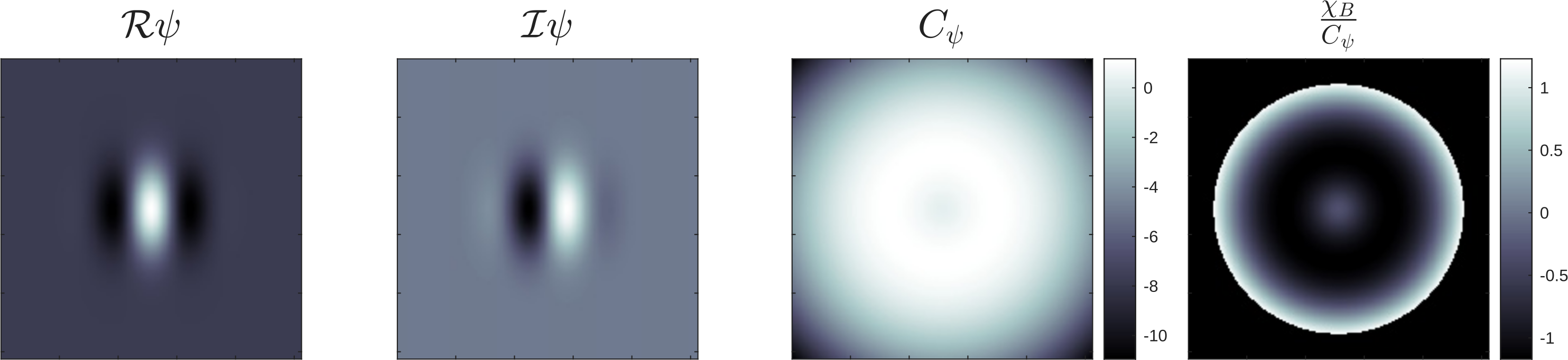}
\caption{Numerical analysis for $p = 1.4, \sigma = 0.225, \varrho = 3, N = 100$.}\label{fig:SIM7}
\end{figure}

Parameters that give negative results are easier to find. If $\varrho$ is much larger than $p$, for example $\varrho = 10$ and $p = 1$, even for a small $\sigma$ (compared to $\frac{1}{L} = 1$) such as $\sigma = 0.1$, that produces large gaussian bells in the Fourier domain, no usable frames are achievable. For these values, $\max_{\omega \in \Omega}n(\omega) = 319$, but, even with $N = 400$ angles, most matrices $G(\omega)$ are practically singular, see Figure \ref{fig:SIM8}.
\begin{figure}[h!]
\centering
\includegraphics[width=.68\textwidth]{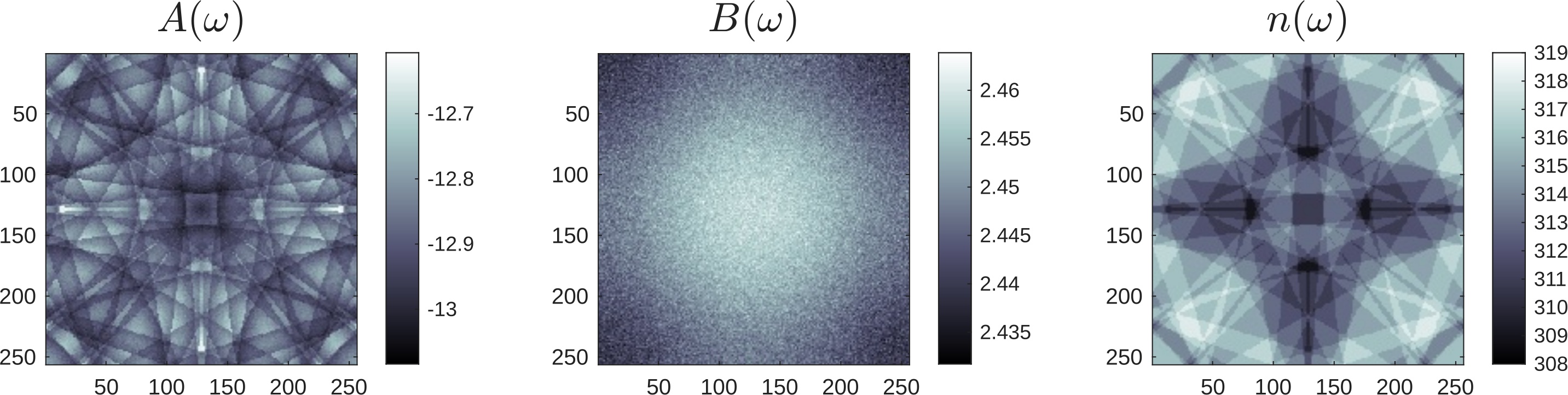}
\includegraphics[width=.68\textwidth]{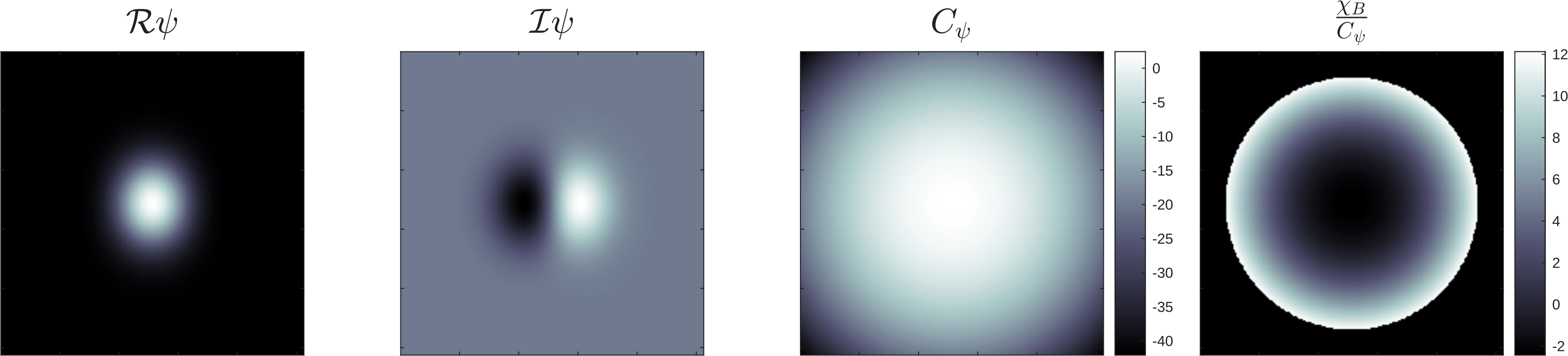}
\caption{Numerical analysis for $p = 1, \sigma = 0.1, \varrho = 10, N = 400$.}\label{fig:SIM8}
\end{figure}

\section{Conclusions}

We have proved with Theorem \ref{th} a characterization of sampling for the $SE(2)$ transform with respect to certain regular sets, that possess some of the attributes of orientation preference maps (OPM) measured in brain's primary visual cortex (V1). This result is analogous to \cite[Th. 2.5]{Bownik2008} and \cite[Th. 4.1]{CabrelliPaternostro2010}. It allows one to decide whether a given set, and a given modulated gaussian mother wavelet, form a frame of a Paley-Wiener space, by solving finite dimensional spectral problems.

\newpage

We have studied numerically those finite dimensional problems for different values of the parameters. When we have stayed sufficiently close to the values provided by the explicit solutions to a simplified problem, we have indeed found frames and quantified their condition number. This allows us to compare the obtained numerical results with some of the literature about the modeling of V1.

In \cite{Bosking2002}, the authors show that in V1 (of tree shrews) the size $\sigma$ of receptive fields is of the same order of magnitude as the typical spacing of OPM, that for our model, in the units of \S \ref{sec:numerics}, is 1. Indeed, we found frames with values of $\sigma$ approximately in between $\frac14$ and $\frac23$. In the neuroscience literature, this is considered instrumental to the coverage principle: it ensures that, for each orientation of a visual stimulus located in a certain position of the visual space, there will be a neuron that is sufficiently close in space and orientation to elicit a positive response, hence to detect the stimulus (see \cite[Figure 1]{Bosking2002}). This mechanism may provide heuristics for the frame problem we considered.

Another comparison of the obtained results with the behavior of V1 concerns the higher spatial density of points observed in OPM with respect to the sampling sets considered here, see Figure \ref{fig:pincenters}. Whilst it is true that more shifted lattices can produce lower condition numbers, as observed in experiments such as that of Figure \ref{fig:SIM7}, in most cases the differences are small. One possible reason for having in V1 so many more approximately shifted lattices than the ones needed to get a reliable representation of orientations, is to allow for the simultaneous distribution of the other features that are indeed collected \cite{Swindale2000}.

Concerning the shifts $\{\alpha_k\}_{k = 1}^N$, it is not clear whether their distributions may influence the condition number: in the proposed setting, we could not observe any significant difference, so we disregarded the possibility and showed averaged results. Appropriate choices of shifts can give rise to distributions that resemble the typical pinwheel figures observed in the V1 of many mammals (see Figure \ref{fig:OPM}), but such distributions may be relevant for other computational or biological purposes.

As a last remark, we note that whilst the conditioning of a linear operation is a classical measure of quality of the inversion in arithmetic systems with relative precision, it is not clear whether the neural code has a precision that can, at least in some limited regimes, be considered relative. Indeed, on the one hand we have Weber's law indicating that psychophysical sensitivity (which is differential in nature) is inversely proportional to the amplitude of the stimulus, and a nonlinearly compressed scaling of numerical information has also been observed at the neural level for the visual representation of numerosities \cite{NiederMiller03}. On the other hand, these may be the effect of collective behaviors, while at the single neuron level the precision may be dictated by spike count (see e.g. \cite[\S 3.4]{DayanAbbott01}), which does not clearly appear to be directly comparable either to relative or to absolute precision.

\end{document}